\documentclass[leqno,11pt,a4paper]{amsart}
\usepackage{graphicx}
\usepackage[usenames,dvipsnames]{color}
\usepackage{hyperref}
\usepackage{ragged2e}
\usepackage{wasysym}
\newtheorem{thm}{Theorem}[section]

\newtheorem{lem}[thm]{Lemma}

\newtheorem{defn}[thm]{Definition}
\newtheorem{conj}[thm]{Conjecture}
\numberwithin{equation}{section}
\theoremstyle{remark}
\newtheorem{rem}[thm]{Remark}
\newenvironment{acknowledge}{\bigskip\noindent\textbf{Acknowledgments.}}{}
\long\def\blankfootnotetext#1{\begingroup\def\thefootnote{\fnsymbol{footnote}}\footnotetext{#1}\endgroup}
\newcommand{\dtemp}{d_\mathrm{b}}
\newcommand{\N}{\mathbb{\Z}_{>0}}
\newcommand{\Z}{\mathbb{Z}}
\newcommand{\Q}{\mathbb{Q}}
\newcommand{\F}{\mathbb{F}}
\newcommand{\Proj}{\mathbb{P}}
\newcommand{\abs}[1]{\left\vert{#1}\right\vert}
\newcommand{\Hom}[1]{\mathrm{Hom}\left({#1}\right)}
\newcommand{\sconv}[1]{\mathrm{conv}\!\left\{{#1}\right\}}
\newcommand{\ip}[1]{\left<#1\right>}
\newcommand{\V}[1]{\mathrm{vert}\!\left({#1}\right)}
\newcommand{\conv}[1]{\mathrm{conv}\!\left({#1}\right)}
\newcommand{\Vol}[1]{\mathrm{Vol}\!\left({#1}\right)}
\newcommand{\widthu}[2]{\mathrm{wd}_{#1}\left({#2}\right)}
\newcommand{\magma}{{\sc Magma}}
\newcommand{\padcol}[1]{\phantom{x}${#1}$\phantom{x}}
\graphicspath{{images/}}
\begin{document}
\author[G.~Brown]{Gavin Brown}
\address{Department of Mathematical Sciences\\Loughborough University\\Loughborough\\LE11 3TU\\United Kingdom}
\email{G.D.Brown@lboro.ac.uk}
\author[A.~M.~Kasprzyk]{Alexander M.~Kasprzyk}
\address{Department of Mathematics\\Imperial College London\\London\\SW7 2AZ\\United Kingdom}
\email{a.m.kasprzyk@imperial.ac.uk}
\blankfootnotetext{2010 {\em Mathematics Subject Classification}: 14G50 (Primary); 52B20, 14M25 (Secondary).}
\title{Small polygons and toric codes}
\begin{abstract}
We describe two different approaches to making systematic classifications of plane lattice polygons, and recover the toric codes they generate, over small fields, where these match or exceed the best known minimum distance. This includes a $[36,19,12]$-code over $\F_7$ whose minimum distance $12$ exceeds that of all previously known codes.
\end{abstract}
\maketitle
\section{Introduction}\label{sec:introduction}
We are interested in planar lattice polygons $P$ -- not least for their own sake -- and, for a finite field $\F_q$, the toric codes $C_P(\F_q)$ that they determine. We recall the construction of such codes, following Hansen~\cite{Han00,Han12}, in Section~\ref{sec:codes} below. Briefly, $P$ and $q$ together determine a $k\times n$ integer matrix -- the toric code -- for integers $k$, the {\em dimension}, and $n$, the {\em block length}. The important invariant of a code $C$ is its {\em minimum distance} $d$, the shortest Hamming distance between any two distinct points of the image of the matrix -- that is, the distance between points of the lattice generated by the rows of the matrix. It is this measure which limits the ability of the code to detect and correct errors.

The collection of all toric codes includes some champions among linear codes, in the sense that they have minimum distance greater than that of any other known code with equal block length and dimension. This is exemplified by the paper of Joyner~\cite{Joy04} and the recent paper of Little~\cite{Lit11}, the latter describing a generalised champion $[49,12,28]$-code over $\F_8$ (we follow the standard notation and denote the invariants of a code by a triple $[n,k,d]$). We find a new champion, a $[36,19,12]$-code over $\F_7$, however the novel part of our paper is the systematic approach to classification of polygons that we use to discover this example. This approach exploits database methods in collaboration with computer-aided polygon computations in a way that seems not to have been pursued in this area. As a spin-off, we classify a certain set of `small' lattice polygons, as we now explain.

One of the key points when considering toric codes over $\F_q$ is to restrict attention to polygons $P$ that lie in a square of side-length $m\le q-2$. In this paper we concentrate on fields $\F_q$ with $q\le9$ a prime power, and so we first enumerate all polygons that lie in a square of side-length $m\le 7$. We regard two polygons $P$ and $P'$ as equivalent if there exists a point $u\in\Z^2$ and change of basis $\varphi$ of the underlying lattice such that $\varphi(P-u)=P'$. Table~\ref{tab:summary} summarises our results by listing the number of integral polygons (up to equivalence) that lie in the square $[0,m] \times [0,m]$, but not in a smaller square. It also lists the maximum number of vertices that such a polygon can have, and the number of polygons that achieve this upper bound. The complete list of polygons is available on the webpage~\cite{GRDb}.

The systematic approach described above would certainly exhaust all toric codes over small fields (even the generalised codes of~\cite{Rua09,Lit11}), but it obviously becomes less practical as $m$ increases. It is remarkable to us that the number of polygons grows as slowly as it does, although this is presumably an artefact of two dimensions. We seek champion codes, that is, codes exceeding the current largest known minimum distances. Computing the minimum distance is very often too hard (it is certainly computable, but in many cases we have reliable estimates of computation time of the order of a year), but determining whether or not a given code is a champion is much easier: in most cases, simple linear combinations of the rows of the corresponding matrix provide sufficiently short vectors to rule it out. We carry out this calculation for each one of these polygon codes for every applicable field of size at most nine. In this early part of the polygonal classification, there is exactly one champion -- the code mentioned above -- as well as many examples that match the current best known minimum distance. All but 257 of the polygonal codes have a short vector that is a combination of at most three rows of the generator matrix; only six of these, other than the champion, require more than four rows.

\begin{table}[htdp]
\caption{The number of polygons (up to equivalence) contained in the square $[0,m]\times[0,m]$, but not contained in a smaller square.}
\centering
\begin{tabular}{|r|ccccccc|}\hline
$m$&\padcol{1}&\padcol{2}&\padcol{3}&\padcol{4}&\padcol{5}&\padcol{6}&\padcol{7}\\\hline
\#Polygons&$2$&$15$&$131$&$1369$&$13842$&$129185$&$1104895$\\
Max vertices&$4$&$6$&$8$&$9$&$10$&$12$&$13$\\
\#Max polygons&$1$&$1$&$1$&$1$&$15$&$2$&$3$\\\hline
\end{tabular}
\label{tab:summary}
\end{table}

The champion polygon and many of the equal champions are LDP-polygons (lattice polygons associated to log del Pezzo surfaces, explained in detail in Section~\ref{sec:ldp}). Motivated by this, we go on to consider codes associated to toric log del Pezzo surfaces, using the classification of all 15,346 LDP-polygons of index $\ell\le 17$~\cite{KKN08}. The difficulty here is that these polygons are not presented as lying in the smallest possible square, so we describe an algorithm that determines the smallest square that contains a polygon equivalent to a given polygon. This algorithm -- an adaptation of the minimum width algorithm -- is the main point, since our search through the resulting data for toric log del Pezzo codes over the smallest possible field $\F_q$ reveals no further champions.

\section{Toric codes}\label{sec:codes}
Let $M\cong\Z^2$ be a two-dimensional lattice, and let $P\subset M\otimes_\Z\Q$ be a convex lattice polygon with vertices in $M$. Let $\F_q$ be a finite field such that, up to translation, $P$ is contained in the square $\sconv{(0,0),(0,q-2),(q-2,0),(q-2,q-2)}$. Hansen demonstrated in~\cite{Han00,Han12} how one can associate a linear code to $P$. Let $\varepsilon\in\F_q$ be a primitive element of the field, and for each lattice point $u:=(u_1,u_2)\in P\cap M$ define:
$$\begin{array}{r@{\ }c@{\ }l}
e(u):\F_q^*\times\F_q^*&\rightarrow&\F_q\\
(\varepsilon^i,\varepsilon^j)&\mapsto&(\varepsilon^i)^{u_1}(\varepsilon^j)^{u_2}.
\end{array}$$
The set of vectors
$$\{(e(u)(\varepsilon^i,\varepsilon^j))_{0\le i,j\le q-2}\mid u\in P\cap M\}$$
then generates a linear code of block length $n=(q-1)^2$ and dimension $k=\abs{P\cap M}$, denoted by $C_P(\F_q)$.

We sketch the connection with toric geometry. For the details see~\cite{Han00,Han02,Joy04}. Associated with any polygon $P$ is a non-singular fan $\Delta$ in $N:=\Hom{M,\Z}$ given by the following construction. Let $h_P:N_\Q\rightarrow\Q$ be the support function defined by
$$h_P(v):=\inf\{\ip{u,v}\mid u\in P\}.$$
This is a piecewise linear function, and partitions $N_\Q$ into a finite collection of strictly convex polyhedral cones; in other words, $h_P$ defines a fan $\Delta'$ in $N_\Q$. We can refine $\Delta'$ to a non-singular fan $\Delta$ by inserting rays via a well-known process (see, for example,~\cite{Oda78}). Notice that, for any cone $\sigma\in\Delta$, there exists a lattice point $l_\sigma\in M$ such that
$$h_P(v)=\ip{l_\sigma,v},\qquad\text{ for all }v\in\sigma.$$

Let $X_P$ be the complete smooth toric surface associated with $\Delta$, where the algebraic torus $T_N$ is defined by $T_N:=\Hom{M,\overline{\F}_q^*}$. The polygon $P$ is associated with the Cartier divisor
$$D_h:=-\!\!\sum_{\text{Rays $\rho$ of $\Delta$}}\!\!\ip{l_\rho,\rho}\cdot\V{\rho},$$
where $\V{\rho}\cong\Proj^1$ is the closure of the $T_N$ orbit of $\rho$ in $X_P$ (again, see~\cite{Oda78} for the details) as follows. The space of global sections $H^0(X_P,\mathcal{O}_{X_P}(D_h))$ has dimension $\abs{P\cap M}$ with basis $\{e(u)\mid u\in P\cap M\}$. Thus the code $C_P(\F_q)$ can be obtained via:
\begin{equation}\label{eq:eval}
\begin{array}{r@{\ }c@{\ }l}
H^0(X_P,\mathcal{O}_{X_P}(D_h))^\mathrm{Frob}&\rightarrow&C_P(\F_q)\\
f&\mapsto&(f(t))_{t\in T_N}.
\end{array}
\end{equation}
The generators of the code are given by the image of the basis; i.e.
$$\{(e(u)(t))_{t\in T_N}\mid u\in P\cap M\}.$$
The toric codes $C_P(\F_q)$ obtained via translation of $P$, or more generally by any affine linear isomorphism of $M\cong\Z^2$, are monomially equivalent~\cite[Theorem~4]{LS07}, hence it is sufficient to consider $P$ up to equivalence. Following~\cite{Rua09,Lit11} one can restrict the evaluation map \eqref{eq:eval} to a subspace of the Riemann--Roch space to construct \emph{generalised} toric codes.

\section{Classifying small polygons}\label{sec:smallpolygons}
\begin{defn}
Let $P\subset M_\Q$ be a lattice polytope, and let $v\in\V{P}$ be a vertex of $P$. If the polytope $P_v:=\conv{(P\cap M)\setminus\{v\}}$ satisfies $\dim{P_v}=\dim{P}$, then $P_v$ is said to have been obtained from $P$ by \emph{shaving}. Given a polytope $Q$ we say that $Q$ can be obtained from $P$ via \emph{successive shaving} if there exists a sequence of shavings $Q=P^{(0)}\subset\cdots\subset P^{(n)}=P$, where $P^{(i-1)}=P^{(i)}_{v_i}$ for some $v_i\in\V{P^{(i)}}$, $0<i\le n$.
\end{defn}

\begin{lem}
Any lattice polygon $Q$ in the box $B_m:=[0,m]\times[0,m]$ can be obtained from $B_m$ via successive shaving.
\end{lem}
\begin{proof}
Let $Q\subset B_m$. If $Q\not=B_m$, then we can certainly shave a vertex $v$ of $B_m$ to obtain a polygon $P_v$ that contains $Q$. Continuing inductively, suppose $P\subset B_m$ is a polygon with $Q\subset P$. If $P\ne Q$ then there exists some vertex $v\in\V{P}$ such that $v\notin Q$ and $Q\subseteq P':=P_v$. But
$$\Vol{Q}\le\Vol{P'}<\Vol{P}.$$
Since these normalised volumes are integers, this process must terminate after finitely many steps.
\end{proof}

\begin{rem}
The normalised volume $\Vol{P}$ of a lattice polygon $P$ is twice the Euclidean volume $\mathrm{vol}(P)$; when $P$ is a lattice polygon, $\mathrm{vol}(P)\in(1/2)\Z$.
\end{rem}

It is clear how this lemma presents an algorithm for classifying small polygons. Choose a box $B_m$ and initialise a sequence containing that box as a polygon:
\vspace{0.25em}
\begin{verbatim}
    Ps:=[Polytope([[0,0],[0,m],[m,0],[m,m]])];
\end{verbatim}
\vspace{0.25em}
The algorithm now extends {\tt Ps} recursively by shaving the vertices from each polygon in the sequence, checking for equivalence before appending each new polygon.
\vspace{0.25em}
\begin{verbatim}
    idx:=1;
    while idx le #Ps do
        for v in Vertices(Ps[idx]) do
            P:=Polytope(Exclude(Points(Ps[idx]),v));
            if Dimension(P) eq 2 and not &or[IsEquivalent(P,Q) : Q in Ps] then
                Append(~Ps,P);
            end if;
        end for;
        idx +:= 1;
    end while;
\end{verbatim}
\vspace{0.25em}
The appeal to {\tt IsEquivalent} is simply checking that the new polygon $P$ is not equivalent to a polygon already contained in the sequence. This is the basic algorithm (after checking simple invariants, attempt to build an isomorphism working up from low-valency vertices), which works well for small values of $m$; larger boxes require some straightforward optimisations that we do not describe.

\section{Log del Pezzo codes}\label{sec:ldp}
A polygon $P\subset N\otimes_\Z\Q$ is said to be an \emph{LDP-polygon} if it contains the origin strictly in its interior, and if each vertex of $P$ is a primitive lattice point in $N:=\Hom{M,\Z}$. In this case the corresponding toric variety is a log del Pezzo surface: it has only cyclic quotient singularities and $-K_X$ is ample. For example, the polygon $\sconv{(1,0),(0,1),(-1,-1)}$ is an LDP-polygon corresponding to the projective plane $\Proj^2$. For such $P$, the dual polygon $P^*$ is defined by
$$P^*:=\left\{v\in M_\Q\mid\ip{v,u}\ge-1\text{ for all }u\in P\right\}.$$
Typically $P^*$ is not a \emph{lattice} polygon, but there exists a smallest integer $\ell\ge 1$ for which the vertices of $\ell P^*$ are lattice points. This integer $\ell$ is called the \emph{index} of $P$. It equals the Gorenstein index of the corresponding log del Pezzo surface.

There are infinitely many LDP-polygons, but finitely many for any fixed index $\ell$. Figure~\ref{fig:523} shows an LDP-polygon with vertices $(-1,2)$, $(1,2)$, $(1,0)$, $(-1,-1)$, $(-2,-1)$. It has index $\ell=10$; the dual is a lattice polygon after dilating by a factor of $10$.

\begin{figure}[htbp]
\centering
\includegraphics[scale=1.1]{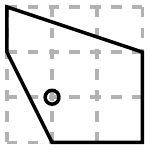}
\qquad
\includegraphics[scale=1.1]{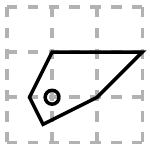}
\caption{An LDP-polygon of index $\ell=10$, with its dual rational polygon.}
\label{fig:523}
\end{figure}

The classification of LDP-polygons of indices $1\le \ell\le 17$ is given in~\cite{KKN08} and available online at~\cite{GRDb}; the example in Figure~\ref{fig:523} is No.~523 in the classification. To use these for toric codes, we allow for translation, and must find equivalent polygons that lie in a box $B_m$, rather than containing the origin as given. The representative polygons in the database are not necessarily the smallest, so we describe an algorithm to determine, for given polygon $P$, the smallest $m$ for which there is a polygon $P'\subset B_m$ equivalent to $P$. (We treat LDP-polygons simply as a source of interesting examples to generate codes, and therefore regard $P$ as lying in $M_\Q$, rather than in terms of algebraic geometry, where they should be viewed as lying in $N_\Q$. Our work here is entirely in terms of $M_\Q$, so no confusion arises from the name.)

\begin{lem}\label{lem:minimum_box}
There is a constructive algorithm that, for any given polygon $P\subset M_\Q$, determines the smallest $m$ for which there is a polygon $P'\subset B_m$ equivalent to $P$, without enumerating any other polygons.
\end{lem}
\begin{proof}
We use a modification of the standard minimal width algorithm. Begin by considering $w_i:=\widthu{e^*_i}{P}$ for each dual basis vector $e^*_i\in N:=\Hom{M,\Z}$. If $w_i\le m$ for all $i$, then, after possible translation, $P$ is contained in $B_m$.

Suppose without loss of generality that $w_1>m$. Translate $P$ so that the (rational) point
$$\frac{1}{\abs{\V{P}}} \sum_{v\in\V{P}}v \in P^{\hbox{\,int}}$$
is at the origin. Let $\Circle\subset P$ be the largest disc, centred at the origin, that is contained in $P$. Then $d:= \widthu{e^*_1}{\Circle}$ satisfies $0< d \le w_1$. In the dual space $N_\Q$, consider the disc $D$ centred at the origin of radius $w_1/d$. We claim that $\widthu{u}{P} > w_1$ for any primitive lattice point $u\in N\setminus D$. Indeed
$$\widthu{u}{P}\ge\widthu{u}{\Circle} > \widthu{u'}{\Circle} = \frac{w_1}d\cdot d = w_1,$$
where $u'$ is the unique rational point lying on the boundary of $D$ in the direction of $u$. Thus $P$ is equivalent to some $P'\subset B_m$ if and only if there exists a dual basis $\mathcal{B}_m$ contained in $D$ such that $\widthu{v}{P}\le m$ for each $v\in\mathcal{B}_m$, and this is a finite search.
\end{proof}

Applying this to the LDP-polygons, we find they distribute over $m$ as in Table~\ref{tab:m}; we have cut the table short, but it could be continued to list all 15,346 polygons, with maximum $m=68$.

\begin{table}[htdp]
\caption{The number of LDP-polygons (up to equivalence) of index $\ell\le 17$ contained in the box $[0,m]\times[0,m]$, but not contained in a smaller box, for small values of $m$.}
\centering
\begin{tabular}{|r|ccccccccc|}\hline
$m$&$2$&$3$&$4$&$5$&$6$&$7$&$8$&$9$&$10$\\\hline
\#&$11$&$62$&$364$&$591$&$1125$&$777$&$1277$&$904$&$1187$\\\hline
\end{tabular}
\label{tab:m}
\end{table}

\section{The computations and champion polygons}\label{sec:champion}
We use the computer algebra system {\magma}~\cite{Magma}, and in particular its convex polytope~\cite{ConvChap} and linear codes packages. The polygons are harvested from the Graded Ring Database~\cite{GRDb}, which can be queried from within {\magma} via an XML interface. Since the number of calculations are rather large, they were run in parallel on a cluster of $144$ processor cores at Imperial College, London.

The LDP-polygon number $75$ in the Graded Ring Database gives a code over $\F_7$ that has length $36$, dimension $19$ and minimum distance $12$ (see Figure~\ref{fig:75}; this corresponds to a unique LDP-polygon, since the only point that could serve as an origin is $(2,2)$). The previous best known minimum distance was $11$, according to Grassl's online database of linear codes~\cite{Grassl} (also accessible from within {\magma}); the theoretical maximum is 15.

\begin{figure}[htbp]
\centering
\includegraphics[scale=1.1]{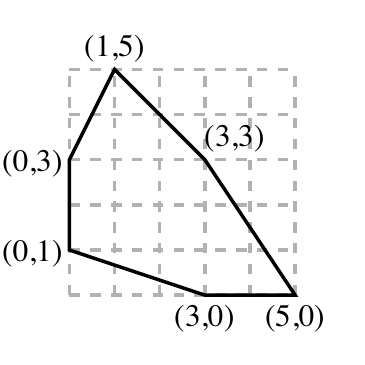}
\caption{The translated LDP-polygon with database id $75$.}
\label{fig:75}
\end{figure}

A curiosity about polygons falls out of the classification. Among those having the maximal number of vertices for their box size, there is a unique one of minimal volume that is homogeneous (by which we mean that all vertices are locally isomorphic to one another). We sketch the polygons with maximal number of vertices for $m\le 7$ in Figure~\ref{fig:max}; notice that for $m=7$ there are two polygons, but only the first is homogeneous.

\begin{figure}[htbp]
\centering
\includegraphics[scale=1.1]{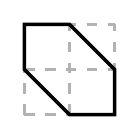}
\includegraphics[scale=1.1]{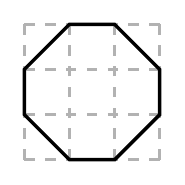}
\includegraphics[scale=1.1]{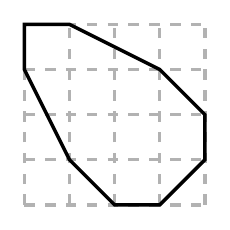}
\includegraphics[scale=1.1]{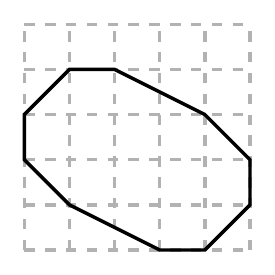}
\includegraphics[scale=1.1]{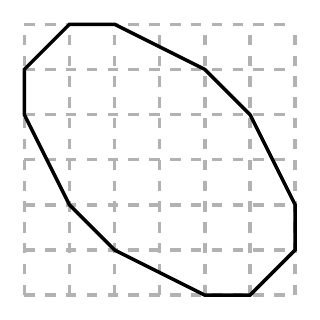}
\includegraphics[scale=1.1]{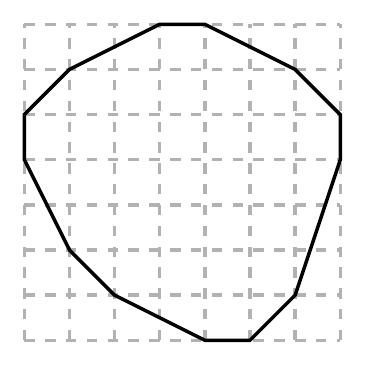}
\includegraphics[scale=1.1]{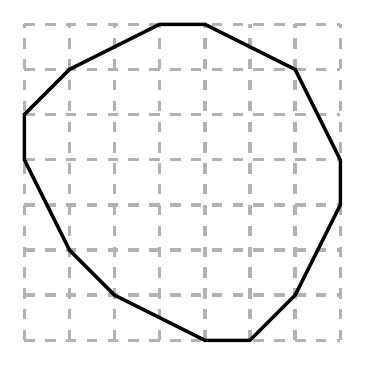}
\caption{The polygons with maximal number of vertices for each box size which have the smallest volume.}
\label{fig:max}
\end{figure}

We reinterpret this geometrically as the following statement, which holds for box size up to seven.
\begin{conj}
Fix $m\in\N$. Consider all toric surfaces $S$ with $-K_S$ ample (but no conditions on the singularities) for which the anticanonical polygon of monomials in $H^0(S,-K_S)$ lies in an $m\times m$ box. Among such $S$, consider those with maximum Picard rank. Then among these
there is a unique surface $S$ which is both smooth and has minimum anticanonical degree.
\end{conj}

\section{Higher prime powers}

We cannot calculate minimum distances, or indeed champions, for higher prime powers: for many polygons with the potential for a large minimum distance, the computation time required to prove this is in the order of millions of years. For specific polygons there exist results on the minimum distance for general prime powers $q$~\cite{LS07,YZ09}, and more can be said when the polygon is Minkowski decomposable~\cite{LS06}, however we do not pursue these techniques here.

Instead, we begin to estimate minimum distances using linear combinations of four rows of the generator matrix. Recall from the introduction that amongst the 1,249,439 toric codes generated
by polygons in boxes of size up to $m=7$, all but six are either shown not to be an equal champion, or to have a prohibitively short vector witnessed by some linear combination of rows involving at most only four rows.

We work with the next two possible values $q=11$ and $13$, with box sizes $8$--$9$ and $10$--$11$ respectively. In each case, we find all LDP-polygons from the database (of index $\le 17$) whose minimum box has size $m$ for which $q$ is minimal satisfying $q-2\ge m$. In other words, each known LDP-polygon will be considered in a minimum box and with respect to its minimum~$q$.

We partition these LDP-polygons by their number of points, $k$. For each $k$, we work through the polygons and find the length of the shortest vector that is a linear combination of up to four rows of the generator matrix, obtaining an upper bound $\dtemp$ on the best minimum distance that could be achieved. These bounds are listed in Tables~\ref{tab:numberq11}--\ref{tab:numberq13}.

\begin{rem}
In theory the bound $\dtemp$ depends on the choice of embedding of $P$ in the box $B_m$; some embeddings may give sharper bounds than others. Lemma~\ref{lem:minimum_box} can generate all possible embeddings in $B_m$, and experimentation suggests that for practical purposes $\dtemp$ is  insensitive to the choices made.
\end{rem}

\begin{table}[htbp]
\caption{The number of LDP-polygonal codes with (minimal) $q=11$. The first row is the dimension $k$ of the code; the second, the length of a known vector, and so an upper bound for the largest minimum distance; the third, the number of polygons (up to equivalence) having such a vector achieving this bound.}\label{tab:numberq11}
\centering
\begin{tabular}{r|cccccccccccccccccc}
$k$&10&11&12&13&14&15&16&17&18&19&20&21&22&23&24&25&26&27\\
$\dtemp$&30&30&30&30&30&30&30&30&30&30&30&40&40&40&30&40&40&30\\
\#&1&1&2&2&3&5&5&4&5&6&6&2&1&1&6&2&1&11\\
\hline
$k$&28&29&30&31&32&33&34&35&36&37&38&39&40&41&42&43&44&45\\
$\dtemp$& 40 & 36&40&40&40&40&40&40&40&40&38&35&35&30&32&39&36&30\\
\# & 1 & 1 &2&5&4&1&2&2&3&2&1&2&1&11&1&1&1&2\\
\hline
$k$&46&47&48&49&50&51&52&53&54&55&56&57&58&59&60&61&62&63\\
$\dtemp$&30&27&27&20&27&27&25&25&18&16&16&12&10&--&--&12&--&8\\
\#&2&4&2&6&1&2&1&1&1&2&1&2&1&--&--&2&--&1\\
\end{tabular}
\end{table}

\begin{table}[htbp]
\caption{The number of LDP-polygonal codes with (minimal) $q=13$; rows as in Table~\ref{tab:numberq11}.}\label{tab:numberq13}
\centering
\begin{tabular}{r|ccccccccccccccccc}
$k$&12&13&14&15&16&17&18&19&20&21&22&23&24&25&26&27&28\\
$\dtemp$ &36&36&36&36&36&36&36&36&36&36&36&36&36&48&36&36&36\\
\# & 1&1&2&2&3&3&4&4&5&4&5&3&3&1&2&1&3\\
\hline
$k$&29&30&31&32&33&34&35&36&37&38&39&40&41&42&43&44&45\\
$\dtemp$&40&36&48&48&48&36&36&40&48&48&48&63&63&63&63&51&54\\
\#&1&2&6&2&2&8&2&1&3&2&7&2&2&1&1&1&1\\
\hline
$k$&46&47&48&49&50&51&52&53&54&55&56&57&58&59&60&61&62\\
$\dtemp$&60&56&60&56&56&45&48&48&48&48&48&40&36&36&36&45&33\\
\#&1&1&1&1&2&2&1&1&3&1&2&3&3&5&1&1&1\\
\hline
$k$&63&64&65&66&67&68&69&70&71&72&73&74&75&76&77&78\\
$\dtemp$&42&40&48&36&30&30&22&12&12&12&33&12&24&12&12&12\\
\#&1&1&1&1&1&1&1&6&3&4&1&3&2&1&3&1\\
\end{tabular}
\end{table}

\begin{acknowledge}
Our thanks to John Cannon for practical help with adapting minimal distance algorithms and for providing {\magma} for use on the Imperial College mathematics cluster, to Andy Thomas for technical assistance, and
to the referees for several interesting suggestions. The second author is supported by EPSRC grant EP/I008128/1, and this research was supported in part by EPSRC Mathematics Platform grant EP/I019111/1.
\end{acknowledge}

\bibliographystyle{amsalpha}
\providecommand{\bysame}{\leavevmode\hbox to3em{\hrulefill}\thinspace}
\providecommand{\MR}{\relax\ifhmode\unskip\space\fi MR }
\providecommand{\MRhref}[2]{%
  \href{http://www.ams.org/mathscinet-getitem?mr=#1}{#2}
}
\providecommand{\href}[2]{#2}

\end{document}